\documentclass[12pt]{article}
\usepackage{amsmath}
\usepackage{amsfonts}
\usepackage{amsthm}
\usepackage{amssymb, setspace}
\usepackage{color,graphicx,epsfig,geometry,hyperref,fancyhdr}
\usepackage[T1]{fontenc}
\usepackage{float}
\usepackage{subfig}
\usepackage{bbm}
\usepackage{mathrsfs,fleqn}
\newtheorem{theorem}{Theorem}[section]

\newtheorem{remark}{Remark}[section]
\newcommand{\N}{\mathbb{N}}

\newcommand{\R}{\mathbb{R}}

\newcommand{\dnu}{\partial_\nu}

\newcommand{\grad}{\nabla}

\newcommand{\ep}{\varepsilon}

\graphicspath{{pics/}}

\begin{document}

\begin{flushleft}
\Large 
\noindent{\bf \Large Regularization of the Factorization Method applied to diffuse optical tomography}
\end{flushleft}

\vspace{0.2in}

{\bf  \large Isaac Harris}\\
\indent {\small Department of Mathematics, Purdue University, West Lafayette, IN 47907 }\\
\indent {\small Email: \texttt{harri814@purdue.edu}}\\


\begin{abstract}
\noindent In this paper, we develop a new regularized version of the Factorization Method for {\color{black} positive} operators mapping a complex Hilbert Space into it's dual space. The Factorization Method uses Picard's Criteria to define an indicator function to image an unknown region. In most applications the data operator is compact which gives that the singular values can tend to zero rapidly which can cause numerical instabilities. The regularization of the Factorization Method presented here seeks to avoid the numerical instabilities in applying Picard's Criteria. This method allows one to image the interior structure of an object with little a priori information in a computationally simple and analytically rigorous way. Here we will focus on an application of this method to diffuse optical tomography where will prove that this method can be used to recover an unknown subregion from the Dirichlet-to-Neumann mapping. Numerical examples will be presented in two dimensions. 
\end{abstract}



\section{Introduction}
In this paper, we focus on two major problems related to {shape reconstruction problems}. The first of which is to derive a theoretically rigorous and computationally simple regularization algorithm for the Factorization Method. Then we will consider an inverse shape problem coming from semiconductor theory see for e.g. \cite{IP-book}. The Factorization Method(see manuscript \cite{kirschbook} for details) falls under the category of {\it qualitative methods}(otherwise known as non-iterative or direct methods) for solving inverse shape problems 
and was introduced in \cite{firstFM}.
 These methods where first introduced in \cite{CK} and are frequently used in non-destructive testing where physical measurements on the surface or exterior of an object is used to determine the integrity of the interior structure. Non-destructive testing plays an important role in many medical and engineering applications. In general, the Factorization Method and similar qualitative methods such as the Direct Sampling Method \cite{DSMdot,DSMHarris,Liu} can be used to derive analytically rigorous and computational simple methods for solving inverse shape problems coming from elliptic \cite{Gebauer}, parabolic \cite{FMheat} and hyperbolic \cite{FM-wave} partial differential equations. One of the main advantages of using qualitative methods over a non-linear optimization techniques is the fact that in general qualitative methods require little a priori information about the region of interest. These methods all give an `indicator' function that can be computed by the given data to reconstruct an unknown region.

The Factorization Method is based on Picard's criteria which would require one to compute a series where one divides by the eigenvalues of a compact operator. Since the eigenvalues of a compact operator can tend to zero rapidly this could result in instabilities in the numerical reconstruction. There has been some previous work on analyzing the use of regularization strategies applied to the Picard's criteria  in \cite{arens2,RegFM}. The analysis studied here is mainly motived by the Generalized Linear Sampling Method introduced in \cite{GLSM} and as well as the similar analysis applied to inverse scattering for near-field data in \cite{Harris-Rome}. Here we will extend the result found in \cite{Harris-Rome} which loosely speaking can  be generalized for a positive compact operator $A: X \to X$ where $X$ is a Hilbert Space with $A=S^* T S$ then 
$$ \ell \in \text{Range} (S^*) \quad \text{ if and only if} \quad \liminf\limits_{\alpha \to 0} \left(x_\alpha  , A x_\alpha \right)_{X} < \infty $$ 
where $x_\alpha$ is the regularized solution to $Ax=\ell$. Here $(\cdot , \cdot )_{X} $ corresponds to the inner-product on $X$. This is proven by appealing to the spectral decomposition of $A$ via the Hilbert-Schmidt Theorem and requiring  standard assumptions on the filter functions coming from classical regularization methods. In this paper, we consider the case for a positive compact operator $A: X \to X^*$ where $X$ is a complex Hilbert Space and $X^*$ is the dual space. We are able to extend the above result to obtain that for  $A=S^* T S$ then 
$$\ell \in\text{Range} (S^*) \quad \text{ if and only if} \quad \liminf\limits_{\alpha \to 0} \langle x_\alpha \, , Ax_\alpha \rangle_{X\times X^*} < \infty.$$
where $x_\alpha$ is the regularized solution to $Ax=\ell$. Here $\langle \cdot \, , \cdot \rangle_{X\times X^*}$ is the sesquilinear dual-pairing between $X$ and $X^*$. The main analytical piece one needs to extend the result is to derive a spectral decomposition for the given operator $A$. This extension is needed if one wishes to apply the `Regularized Factorization Method' to problems coming from Diffuse Optical and Electrical Impedance Tomography where the operator $A$ is given by the Dirichlet-to-Neumann operator which maps $H^{1/2}(\Gamma)$ into it's dual space $H^{-1/2}(\Gamma)$ where $\Gamma$ is some curve/surface.

The rest of the paper is structured as follows. In Section \ref{RFM}, we develop the analytic framework for the regularized variant of the Factorization Method. To do this, we derive a spectral decomposition for a positive compact operator  $A: X \to X^*$ and derive a regularized version of Picard's criteria. Then in Section \ref{InvShape}, we consider an inverse shape reconstruction problem for a model problem related to semiconductor theory where the given data operator is the Dirichlet-to-Neumann mapping. In Section \ref{numerics}, numerical examples in two dimensions are presented for solving the inverse shape problem. Finally, in Section \ref{end} we summaries the results in the paper as well as discuss future research for this methods. 

\section{Regularized Factorization Method}\label{RFM}
In this section, we wish to derive a new Regularized Factorization Method that can be used to solve inverse shape problems. The analysis in this section is motivated by the works in \cite{arens,arens2,GLSM} where other regularization techniques applied to the Factorization Method are studied. In \cite{arens,arens2} the Linear Sampling Method which is the predecessor to the Factorization Method is validated by appealing to the regularized solution to the far-field equation. Recently, in \cite{GLSM} a new regularization technique was studied for the Linear Sampling Method to solve the far-field equation which incorporates the analysis from the Factorization Method. We now develop a simpler regularization strategy for a positive compact operator to determine a range characterization. The method presented here does not require the minimization of a complex functional and gives a rigors range characterization.

To begin, we assume that $A: X \to X^*$ with complex separable Hilbert Space $X$ is a positive compact operator. Here we let $X^*$ denote the dual space of $X$. Throughout this section, we will denote $\langle \cdot \, , \cdot \rangle_{X\times X^*}$ as the sesquilinear dual-product between $X$ and $X^*$. Furthermore, assume that $H$ is the separable Hilbert pivoting space to the dual-product $\langle \cdot \, , \cdot \rangle_{X\times X^*}$ such that it coincides with the inner-product on $H$ with dense inclusions $X \subset H \subset X^*$(i.e. a Gelfand triple). Furthermore, assume that we have the factorization 
$$A=S^* T S  \quad \text{ where } \quad S: X \to V  \quad \text{ and } \quad T: V \to V$$
with $V$ also being a Hilbert Space. Here the adjoint for $S$ is a mapping $S^*: V \to X^*$ given by the equality 
\begin{align}\label{adjoint}
(S x , v )_{V} = \langle x , S^* v \rangle_{X\times X^*} \quad \text{ for all } \quad  v \in {V} \textrm{ and } x \in X.
\end{align}
We will assume that $S$ is a compact and injective operator as well as assuming $T$ is a bounded operator and coercive on Range$(S)$ i.e. there is a fixed constant $\beta >0$ such that $\beta \| Sx \|^2_V \leq (Sx , TSx)_V $ for all $x \in X$.

Now, to continue we consider the case when we also have $A=Q^* Q$ where $Q: X \to H$ and it's adjoint $Q^*: H \to X^*$ are bounded linear operators. Similarly  the adjoint $Q^*$ is defined by  
$$(Q x , \phi )_{H} = \langle x , Q^* \phi \rangle_{X\times X^*} \quad \text{ for all } \quad  \phi \in {H} \textrm{ and } x \in X.$$ 
The operator $A$ also has this factorization when it is positive and self-adjoint where $Q$ is it's square root. To derive our regularized Factorization Method we will characterize the range of $S^*$ with the regularized solution to $Ax=\ell$ for any $\ell \in X^*$. To do so, we first connect the characterize the range of $S^*$ to the range of $Q^*$.


\begin{theorem}\label{range-SQ}
Let $A: X \to X^*$ with factorizations $A=Q^* Q$ and $A=S^* T S$ such that $S: X \to V$,  $T: V \to V$ and $Q: X \to H$  are bounded operators where $X$, $V$ and $H$ are Hilbert Spaces. Assume that $T$ is coercive on Range$(S)$ then we have that 
$$\text{Range}(Q^* )=\text{Range}(S^* ).$$
\end{theorem} 
\begin{proof}
To prove the claim, we notice that by appealing to the factorizations $A=Q^* Q$ and $A=S^* T S$ we have that 
$$ \| Qx \|^2_H = \langle x \, , Ax \rangle_{X\times X^*} = (Sx , TSx)_V \quad \text{ for all } \quad  x \in X.$$
Note that we have also used the definition of the adjoints for the operators $S: X \to V$ and $Q: X \to H$. Due to the coercivity of  $T$ there is constant $\beta >0$ such that 
$$ \beta \| Sx \|^2_V  \leq \| Qx \|^2_H \leq \| T \| _{V\to V} \| Sx \|^2_V \quad \text{ for all } \quad  x \in X.$$
Therefore, by Theorem 1 of \cite{range-lemma} we can conclude that Range$\big( Q^* \big)$ = Range$(S^*)$.
\end{proof}

In order to continue we derive a spectral decomposition for the positive compact operator $A$ as is done in \cite{kirschpp}. To this end, notice that the Riesz Representation Theorem implies that there is a bijective isometry $J : X^* \to X$ such that 
$$ \ell \longmapsto x_\ell \quad \text{ where } \quad  (x \, , x_\ell )_X = \langle x \, , \ell \rangle_{X\times X^*}  \quad \text{ for all } \quad  x \in X.$$
Note that, due to the fact that we have a sequilinear dual-product gives that $J$ is linear. Now, define the compact operator $J A : X \to X$ and notice that 
$$\big(x \, , (JA) x \big)_X = \langle x \, , A x \rangle_{X\times X^*} > 0 \quad \text{ for all } \quad  x \in X \setminus \{0\}$$
since $A$ is assumed to be positive. This implies that $JA$ is a self-adjoint(see Corollary 7.3 of \cite{LA-done}) compact operator on the complex Hilbert Space $X$. 
Therefore, by the Hilbert-Schmidt Theorem we have the there exists an eigenvalue decomposition 
\begin{align}\label{LAdecomp}
\{ \lambda_n ; x_n \}_{n \in \N} \in \R_{>0} \times X\quad \text{ such that} \quad (JA) x = \sum \lambda_n (x,x_n)_X  \, x_n
\end{align}
where $\lambda_n$ is a decreasing sequence converging to zero and $x_n$ is an orthonormal basis of $X$ since $JA$ is injective. We now define $\ell_n \in X^*$ to be the unique solution to $J \ell_n = x_n$ for any $n \in \N$. The functionals $\ell_n \in X^*$ satisfy 
\begin{align}\label{dual-basis} 
\langle x_m \, , \ell_n \rangle_{X\times X^*} = (x_m \, , x_n )_X = \delta_{mn} \quad \text{ for any } \quad n,m \in \N
\end{align}
and are the corresponding dual-basis of $X^*$.  
{\color{black}
\begin{remark} Note, that in order to remove the assumption that $X$ is a complex Hilbert Space we could assume that the bilinear form
$$(x,y) \longmapsto  \langle y \, , A x \rangle_{X\times X^*} $$
is symmetric. This would again give that $JA$ is a self-adjoint compact operator acting on $X$. Which is the key piece needed to obtain equation \eqref{LAdecomp} which is the only thing in the analysis that requires the complex assumption on the Hilbert Space.
\end{remark} 
} 

We now show that the set $\{\ell_n \}_{n \in \N}$ is a complete orthonormal set for $X^*$ as well as determine a representation for any $\ell \in X^*$ and the operator $A$. Notice, that $X^*$ is a Hilbert Space with inner-product defined by the sesquilinear form
$$ (\ell , \varphi )_{X^*} = (x_\ell , x_\varphi )_{X} \quad \text{ for all } \quad \ell, \varphi \in X^*  \quad \textrm{ where }\;  J \ell = x_\ell \; \textrm{ and } \; J \varphi = x_\varphi.$$
From the definition of the inner-product  we have  that 
$$(\ell_m , \ell_n )_{X^*}  =  (x_m \, , x_n )_X = \delta_{mn} \quad \text{ for any } \quad n,m \in \N.$$
This implies that $\{\ell_n \}_{n \in \N}$ is an orthonormal set in $X^*$. In order to prove that the set is complete we assume that $\ell \in X^*$ is orthogonal to the set $\{\ell_n \}_{n \in \N}$. Therefore, we have that for all $n \in \N$
$$ 0 = (\ell , \ell_n )_{X^*} =  (x_\ell \, , x_n )_X  \quad \text{ giving that } \quad x_\ell = 0$$
and we can conclude that $\ell = 0$ since $J$ is an isometry. This implies that the sequence $\ell_n$ forms an orthonormal basis on $X^*$. We then conclude
$$\ell = \sum   (\ell , \ell_n )_{X^*} \, \ell_n \quad \text{ and notice that  } \quad  \overline{ \langle x_n,\ell  \rangle}_{X\times X^*} = (\ell , \ell_n )_{X^*}.$$
From this we have the representation
\begin{align}\label{elldecomp}
\ell = \sum  \overline{ \langle x_n,\ell  \rangle}_{X\times X^*} \, \ell_n \quad \text{ for all } \quad  \ell \in X^*.
\end{align}
By the injectivity of the operator $J$ we can conclude that the operator $A$ has the spectral decomposition 
\begin{align}\label{Adecomp}
A x = \sum \lambda_n  \langle x,\ell _n \rangle_{X\times X^*} \,  \ell_n \quad \text{ for all } \quad  x \in X.
\end{align}
To derive our regularized variant of the Factorization Method we will characterize the range of $S^*$ by the spectral decomposition of the compact operator $A$ given by \eqref{Adecomp} which is given in the following result.

\begin{theorem}\label{range-id}
Let $A: X \to X^*$ {\color{black} be a positive operator} with the factorization $A=S^* T S$ such that $S: X \to V$ and $T: V \to V$ are bounded linear operators where $X$ and $V$ are Hilbert Spaces. Assume that $S$ is compact and injective as well as $T$ being coercive on Range$(S)$. Then we have that 
$$\ell \in \text{Range}(S^* ) \quad \text{if and only if} \quad  \sum \frac{1}{\lambda_n}  \left|\langle x_n ,\ell \rangle_{X\times X^*} \right|^2 <\infty$$
where $\{ \lambda_n \, ; x_n \}_{n \in \N} \in \R_{>0} \times X$ are given by the spectral decomposition  \eqref{Adecomp} of $ A$.
\end{theorem} 
\begin{proof}
By our assumptions on $T$ and $S$ we have that $A$ is a positive compact operator. This implies that $A$ has the spectral decomposition \eqref{Adecomp} where $\{x_n\}_{n \in \N}$ is an orthonormal basis of $X$. Therefore, we denoted by $\{\phi_n \}_{n \in \N} $ an orthonormal basis of $H$. Now, define the bounded linear operator 
$Q^*: H \to X^*$ such that 
$$Q^* \phi = \sum \sqrt{\lambda_n}  ( \phi, \phi_n )_{H} \,  \ell_n \quad \text{ for all } \quad  \phi \in {H}.$$
Then recall that we define the operator $Q : X \to H$ by the equality
$$ (Q x , \phi)_{H} = \langle x , Q^* \phi \rangle_{X\times X^*} \quad \text{ for all } \quad  \phi \in {H} \textrm{ and } x \in X.$$
Therefore, we can conclude that for any $n \in \N$
$$ (Q x , \phi_n )_{H} = \langle x , Q^*\phi_n \rangle_{X\times X^*} = \sqrt{\lambda_n} \langle x ,\ell_n \rangle_{X\times X^*} $$
where we have used the fact that  $\{\phi_n \}_{n \in \N}$ is an orthonormal set and the definition of $Q^*$. This implies that the adjoint operator can be expressed as 
$$Q x = \sum \sqrt{\lambda_n} \langle x ,\ell_n \rangle_{X\times X^*} \,  \phi_n \quad \text{ for all } \quad  x \in X.$$
We can now conclude that $Q^*\phi_n=\sqrt{\lambda_n} \ell_n$ and $Q x_n=\sqrt{\lambda_n} \phi_n$ by \eqref{dual-basis} for any $n \in \N$. This gives that $Q^* Q x_n = \lambda_n \ell_n$ for any $n \in \N$. Now, notice that by \eqref{dual-basis} and \eqref{Adecomp} we have that $Ax_n = \lambda_n \ell_n = Q^* Qx_n$ which implies that $A = Q^* Q$ since they agree on a basis.  Therefore, by Theorem \ref{range-SQ} we can conclude that Range$(S^*)$ = Range$( Q^*)$.

Using the definition of $Q^*$, we can proceed as in the proof of Picard's Criteria (see for e.g. Theorem 1.28 of \cite{TE-book}) to prove the result. Now, we have that $\ell \in$ Range$(S^*)$ which implies that  $\ell \in$ Range$(Q^*)$ and is equivalent to $Q^*\phi = \ell$ for some $\phi \in {H}$. By appealing to the definition of $Q^*$ and \eqref{elldecomp} we have that 
$$\overline{ \langle x_n,\ell  \rangle}_{X\times X^*} =  \sqrt{\lambda_n}  ( \phi, \phi_n )_{H} \quad \text{ for all } \quad n \in \N $$ 
by the linear independence of the set $\{\ell_n \}_{n \in \N}$. Since $\{\phi_n \}_{n \in \N}$ is an orthonormal basis of ${H}$ we conclude that 
$$\sum \frac{1}{\lambda_n}  \left|\langle x_n ,\ell \rangle_{X\times X^*} \right|^2 = \|\phi\|^2_{H} < \infty$$
proving the claim.
\end{proof}

Now that we have characterized the range of $S^*$ via the spectral decomposition \eqref{Adecomp} we now wish to develop a similar result as in \cite{GLSM}. The method we present here has the same theoretical result but one does not need to minimize a non-convex cost functional. To this end, we first notice that if $\ell \in$Range$(A)$ then we have that $Ax=\ell$ for some $ x \in X.$ Then by appealing to \eqref{elldecomp} and \eqref{Adecomp} we obtain the formula  
$$ x =  \sum \frac{1}{\lambda_n} \overline{ \langle x_n,\ell  \rangle}_{X\times X^*} \, x_n.$$
Since $A$ is positive and compact we have that $\lambda_n >0$ and tends to zero as $n \to \infty$. Therefore, we define $x_\alpha$ to be the regularized solution of $Ax=\ell$ which is given by 
\begin{align}\label{regsolu}
x_\alpha =  \sum \frac{f_\alpha(\lambda_n)}{\lambda_n} \overline{ \langle x_n,\ell  \rangle}_{X\times X^*} \, x_n.
\end{align}
Here the $f_\alpha (t)$ denotes the filter function associated with a specific regularization scheme. Notice, that we have used the fact that $\{\lambda_n ; x_n ; \ell_n \} \in \R_{>0}\times X\times X^*$ is the singular system for $A$.  For all $\alpha>0$ the filter $f_\alpha (t) : \big(0 ,\lambda_1 \big] \to \R_{\geq0}$  is assumed to be a family of functions that satisfies for $0 < t \leq  \lambda_1$
$$\lim\limits_{\alpha \to 0} f_\alpha(t) =1 \quad \text{ and } \quad  f_\alpha(t) \leq C_{\text{reg}} \quad \text{for all} \,\, \alpha>0$$
where $\lambda_1$ corresponds to the largest spectral value defined in \eqref{LAdecomp}. 
Two common filter--functions are defined as  
\begin{align}
f_\alpha (t) =  \frac{t^2}{t^2+\alpha} \quad \text{and} \quad  \displaystyle{ f_\alpha (t)= \left\{\begin{array}{lr} 1 \, \, & \, \, t^2\geq \alpha,  \\
 				&  \\
 0\, \,  & \, \,   t^2 < \alpha.
 \end{array} \right.}  \label{filters}
 \end{align}
The filter functions in \eqref{filters} are for Tikhonov regularization and Spectral cutoff, respectively(see for e.g. \cite{kirschipbook}). 
Now, just as in \cite{Harris-Rome} we can now derive a regularization variant of the Factorization Method. The results here extend the analysis in \cite{Harris-Rome} to positive compact operators mapping a Hilbert Space into it's corresponding dual space. The following results connects the Range$(S^*)$ to the boundedness as the regularization parameter $\alpha \to 0$ of the quantity $\langle x_\alpha \, , Ax_\alpha \rangle_{X\times X^*}$.

\begin{theorem}\label{reg-range-lemma}
Let $A: X \to X^*$ {\color{black} be a positive operator} with the factorization $A=S^* T S$ such that $S: X \to V$ and $T: V \to V$ are bounded linear operators where $X$ and $V$ are Hilbert Spaces. Assume that $S$ is compact and injective as well as $T$ being coercive on Range$(S)$.  Then we have that
$$\ell \in\text{Range} (S^*) \quad \text{ if and only if} \quad \liminf\limits_{\alpha \to 0} \langle x_\alpha \, , Ax_\alpha \rangle_{X\times X^*} < \infty$$
where $x_\alpha$ is the regularized solution given by \eqref{regsolu} to $Ax=\ell$. 
\end{theorem} 
\begin{proof}
We first notice that $A$ is a positive compact operator which gives that it has the spectral decomposition \eqref{Adecomp} where  $Ax_n = \lambda_n \ell_n$ for any $n \in \N$ and we have that 
$$ Ax_\alpha =  \sum {f_\alpha(\lambda_n)} \overline{ \langle x_n,\ell  \rangle}_{X\times X^*} \, \ell_n.$$
Then, we can obtain the equality  
$$\langle x_\alpha \, , Ax_\alpha \rangle_{X\times X^*} = \sum \frac{f^2_\alpha(\lambda_n)}{\lambda_n} | \langle x_n,\ell  \rangle_{X\times X^*}|^2$$
by the  sesquilinear definition of the dual-product as well as the duality relationship given in \eqref{dual-basis}. In order to prove the claim we now bound the limiting value of the quantity $\langle x_\alpha \, , Ax_\alpha \rangle_{X\times X^*}$ by the Picard's Criteria from Theorem \ref{range-id}. 

To this end, we now use the fact that the family of filter functions $\{f_\alpha\}_{\alpha> 0}$  is uniformly bounded to obtain the estimate  
$$\langle x_\alpha \, , Ax_\alpha \rangle_{X\times X^*} \leq C^2_{\text{reg}} \sum \frac{1}{\lambda_n} | \langle x_n,\ell  \rangle_{X\times X^*}|^2.$$
We now find a lower bound for the quantity $\langle x_\alpha \, , Ax_\alpha \rangle_{X\times X^*}$. Now, notice that for any $N \in \N$ we have that 
$$\langle x_\alpha \, , Ax_\alpha \rangle_{X\times X^*} \geq \sum\limits_{n=1}^{N} \frac{f^2_\alpha(\lambda_n)}{\lambda_n} | \langle x_n,\ell  \rangle_{X\times X^*}|^2$$
and by taking the $\liminf$ as $\alpha \to 0$ we obtain that 
$$ \liminf\limits_{\alpha \to 0} \langle x_\alpha \, , Ax_\alpha \rangle_{X\times X^*} \geq \sum\limits_{n=1}^{N} \frac{1}{\lambda_n} | \langle x_n,\ell  \rangle_{X\times X^*}|^2 \quad \text{ for all } \quad N \in \N $$
since the filter function satisfies $f^2_\alpha(\lambda_n) \xrightarrow{\alpha \to 0}  1$ for all $n \leq N$. By summing to infinity in the above inequality gives 
$$ \sum\frac{1}{\lambda_n} | \langle x_n,\ell  \rangle_{X\times X^*}|^2 \leq \liminf\limits_{\alpha \to 0}   \langle x_\alpha \, , Ax_\alpha \rangle_{X\times X^*} \leq C^2_{\text{reg}} \sum \frac{1}{\lambda_n} | \langle x_n,\ell  \rangle_{X\times X^*}|^2.$$ 
Appealing to Theorem \ref{range-id} proves the claim.  
\end{proof}

We note that the filter functions defined in \eqref{filters} satisfy that $C_{\text{reg}}=1$. Therefore, the proof of Theorem \ref{reg-range-lemma} implies that 
$$ \lim\limits_{\alpha \to 0} \langle x_\alpha \, , Ax_\alpha \rangle_{X\times X^*} = \sum \frac{1}{\lambda_n} | \langle x_n,\ell  \rangle_{X\times X^*}|^2.$$ 
The above result gives a range identity theorem which is a way to characterize the Range$(S^{*})$ to the regularized solution of $Ax=\ell$. This range characterization will be used in the preceding  section to develop a sampling algorithm applied to an inverse shape problem in Diffuse Optical Tomography. \\

{\color{black}
\noindent{\it Connection to the Generalized Linear Sampling Method:}\\
 The connection between the regularization of $Ax=\ell$ and the range of $S^*$ was also studied in \cite{GLSM}. The method presented in \cite{GLSM} is referred to the Generalized Linear Sampling Method(GLSM). This is due to the fact that it connects the classical Linear Sampling Method and the Factorization Method(see also \cite{arens}). This is done by considering the minimizer to the functional 
$$\mathcal{J}_{\alpha} \big(\ell ; x \big)  = \alpha   \langle x \, , Ax \rangle_{X\times X^*}  +\| Ax - \ell  \|^2_{X^*}. $$
Due to the assumptions in Theorem \ref{reg-range-lemma} we have that for every $\ell \in X^*$ that $\mathcal{J}_{\alpha} \big(\ell ; x \big)$ has a unique minimizer. Indeed,
 using the decompositions in equations \ref{elldecomp} and \ref{Adecomp} we have that the functional can be written as   
$$\mathcal{J}_{\alpha} \big(\ell ; x \big)  =  \alpha \sum\lambda_n  | \langle x ,\ell_n \rangle_{X\times X^*}|^2 + \sum | \lambda_n \langle x ,\ell_n \rangle_{X\times X^*} - \overline{ \langle x_n ,\ell \rangle}_{X\times X^*} |^2.$$
Therefore, 
arguing as in \cite{GLSM} we have that the minimizer $x_\alpha$ is given by 
$$x_\alpha =\sum \frac{\lambda_n}{\alpha \lambda_n +\lambda_n^2}  \overline{ \langle x_n ,\ell \rangle}_{X\times X^*} x_n$$
for any $\ell \in X^*$. Then the GLSM gives the result that
$$\ell \in\text{Range} (S^*) \quad \text{ if and only if} \quad \liminf\limits_{\alpha \to 0} \langle x_\alpha \, , Ax_\alpha \rangle_{X\times X^*} < \infty$$
where $x_\alpha$ is the minimize of $\mathcal{J}_{\alpha} \big(\ell ; x \big)$ provided that 
$$ \mathcal{J}_{\alpha} \big(\ell ; x_\alpha \big)  \leq  \inf\limits_{{\bf \ell }} \mathcal{J}_{\alpha} \big(\ell; x \big)  +C \alpha $$
for some $C>0$ independent of $\alpha$(see Theorem 3 of \cite{GLSM}). Notice that our calculations imply that the filter function for the GLSM is given by 
$$f_\alpha (t) =  \frac{t}{\alpha  + t}$$ 
and notice that for all $t>0$
$$\lim\limits_{\alpha \to 0} f_\alpha(t) =1 \quad \text{ and } \quad  f_\alpha(t) \leq 1 \quad \text{for all} \,\, \alpha>0.$$
This implies that, for this problem the GLSM is covered by the theory presented in this section. This follows from the fact that computing the minimizer of $\mathcal{J}_{\alpha} \big(\ell ; x \big)$ is just a regularization scheme for solving $Ax=\ell$ under our assumptions. 
}

\section{An Application to Diffuse Optical Tomography}\label{InvShape}
In this section, we will consider a shape reconstruction problem. The problem comes from semiconductor theory where one wishes to use boundary measurements to determine existence of an interior structure. The idea is to use the known/measured Dirichlet-to-Neumann mapping which in general uniquely determines unknown subregions in Diffuse Optical Tomography \cite{DOTuniqueness}. Here we will assume that the domain $D \subset \R^d$ (for $d=2,3$) is a bounded simply connected open set with Lipschitz boundary $\Gamma_\text{1}$ with unit outward normal $\nu$. Now, we let $D_0 \subset D$ (with possible multiple components) be a connected open sets with Lipschitz boundary. We shall also assume that there is a fixed constant $\delta>0$ such that $\text{dist}(\Gamma_\text{1} , \overline {D}_0)= \delta.$ 

Now, let $u \in H^1(D)$ be the unique solution to 
\begin{align}
-\Delta u +\chi({D}_0) u = 0 \,\, \text{ in } \,\, D \quad \text{ and } \quad u|_{_{\Gamma_{1}} } = f \label{prob1}
\end{align} 
for any given $f \in H^{1/2}(\Gamma_\text{1})$ where $\chi(\cdot)$ denotes the indicator function. One can easily show that \eqref{prob1} is well-posed by considering it's equivalent variational formulation(see for e.g. \cite{evans}). This implies that the {\it Dirichlet-to-Neumann} (DtN) mapping
$$ \Lambda : H^{1/2}(\Gamma_{1}) \longrightarrow H^{-1/2}(\Gamma_{1})   \quad \text{such that} \quad \Lambda f = \partial_{\nu} u|_{_{\Gamma_{1}} }$$
is a bounded linear operator by Trace Theorems.  {\color{black} Notice that, in general this application would result in real valued solutions $u$ as long as the boundary data $f$ is real valued. To keep all generality we will proceed as if the functions are complex valued but all of the results still hold for real valued functions.}

The goal is to now use the theory developed in Section \ref{RFM} to derive a sampling algorithm to recover the unknown interior shape $D_0$ from the knowledge of $\Lambda$. This problem is similar to the ones considered in \cite{DSMdot,Gebauer,DOT-RtR} where different sampling methods are employed to solve the inverse shape problem for recovering `defects' in an object. Similar to the previous works we consider the difference of the DtN mappings $(\Lambda - \Lambda_{0})$ where $u_{0} \in H^1(D)$ is the Harmonic Lifting of $f\in H^{1/2}(\Gamma_\text{1})$ satisfying  
\begin{align}
\Lambda_0 f = \partial_{\nu} u_0|_{_{\Gamma_{1}} }  \quad \text{such that} \quad \Delta u_0= 0 \,\, \text{ in } \,\, D \quad \text{ and } \,\, u_0 |_{_{\Gamma_{1}} } = f. \label{harmonic}
\end{align}
Again, we have that $\Lambda_0 : H^{1/2}(\Gamma_{1}) \longrightarrow H^{-1/2}(\Gamma_{1})$ is a bounded linear operator. 

We will now derive a factorization for the difference of the DtN mappings $(\Lambda - \Lambda_0)$. To this end, notice that 
$$ -\Delta (u-u_0) = - \chi({D}_0) u  \,\, \text{ in } \,\, D \quad \text{ and } \quad (u-u_0) |_{_{\Gamma_{1}} } = 0.$$
From this we define $w \in H^1_0(D)$ to  be the unique solution to the source problem 
\begin{align}
 -\Delta w =  \chi({D}_0) h  \,\, \text{ in } \,\, D \quad \text{ and } \quad w |_{_{\Gamma_{1}} } = 0 \label{prob2}
\end{align} 
for any given $h \in L^2(D_0)$ which is well-posed by the variational formulation. Now, define the Source-to-Neumann operator for \eqref{prob2} such that  
$$G: L^2(D_0) \longrightarrow H^{-1/2}(\Gamma_1) \quad \text{given by} \quad Gh = \partial_{\nu} w \big|_{\Gamma_1}.$$
Therefore, by the well-posedness of \eqref{prob2} we have that $ \partial_{\nu} w|_{_{\Gamma_{1}} } = (\Lambda - \Lambda_0) f$ provided that $h= - u|_{D_0}$. From this we define the solution operator for \eqref{prob1} such that 
$$ S: H^{1/2}(\Gamma_1) \longrightarrow L^2(D_0) \quad \text{ given by } \quad Sf = u|_{D_0}.$$
From this we see that $(\Lambda - \Lambda_0)f = -GSf$ for any $f \in H^{1/2}(\Gamma_{1})$. Now, in order to factorize the operator $(\Lambda - \Lambda_0)$ further we need to compute the adjoint of the solution operator $S$. To do so, we define the sesquilinear dual-product 
$$\langle \varphi  ,\psi  \rangle_{\Gamma_1} = \int\limits_{\Gamma_1 } \varphi  \,\overline{ \psi }  \, \text{d}s \quad \text{ for all } \quad \varphi \in H^{1/2}(\Gamma_1) \quad \text{ and }  \quad  \psi \in  H^{-1/2}(\Gamma_1)$$ 
between the Hilbert Space $H^{1/2}(\Gamma_1)$ and it's dual space $H^{-1/2}(\Gamma_1)$ where $L^2(\Gamma_1)$ is the Hilbert pivot space. Recall, that we have 
$$ H^{1/2}(\Gamma_1) \subset L^{2}(\Gamma_1) \subset H^{-1/2}(\Gamma_1)$$ 
with dense inclusions. The adjoint operator $S^*$ is given in the following result. 
\begin{theorem}\label{S-adjoint}
The adjoint operator $S^{*}: L^2(D_0) \longrightarrow H^{-1/2}(\Gamma_1)$ is given by 
$$S^{*} g = - \partial_{\nu} v \big|_{\Gamma_1} \quad \text{where} \quad -\Delta v +\chi({D}_0) v = \chi({D}_0)g \,\, \text{ in } \,\, D \quad \text{ and } \quad  v|_{_{\Gamma_{1}} } = 0.$$
Moreover, the operator $S$ is compact and injective. 
\end{theorem}
\begin{proof}
Notice, that by a variational argument we have that the solution $v \in H^1_0(D)$ exists and depends continuously on $g \in L^2(D_0)$. Now we compute the adjoint operator $S^*$. Therefore, by Green's 2nd Theorem we have that 
$$\int\limits_{\Gamma_1} \overline{v} \dnu u-u \dnu \overline{v} \, \text{d}s = \int\limits_{D} \overline{v} \Delta u-u \Delta \overline{v} \, \text{d}x.$$
First we consider the boundary integral in the above equality. By the definition of $u$ and $v$ we have that 
$$\int\limits_{\Gamma_1} \overline{v} \dnu u-u \dnu \overline{v} \, \text{d}s = - \int\limits_{\Gamma_1} f  \dnu \overline{v} \, \text{d}s$$
where we have used the boundary conditions on $\Gamma_1$ for $u$ and $v$. Now, for the volume integral we notice that we have 
$$  \int\limits_{D} \overline{v} \Delta u-u \Delta \overline{v} \, \text{d}x= \int\limits_{D} \big(\chi(D_0)\overline{v} -\Delta \overline{v} \big)u \, \text{d}x=\int\limits_{D_0}u \overline{g}\, \text{d}x.$$
where we have used the partial differential equations that $u$ and $v$ in the interior of $D$.
From this we have that 
$$\langle f  ,S^* g  \rangle_{\Gamma_1} = - \int\limits_{\Gamma_1} f  \dnu \overline{v} \, \text{d}s =  \int\limits_{D_0}u \overline{g}\, \text{d}x = (Sf,g)_{L^2(D_0)}$$
for all $f \in H^{1/2}(\Gamma_1)$ and $g \in L^2(D_0)$ which implies that $S^{*} g = - \partial_{\nu} v \big|_{\Gamma_1}$. 

To prove injectivity we let $Sf=0$ which means that $u=0$ in $D_0$ which implies that $u$ has zero Cauchy data on $\partial D_0$. Then, appealing to Holmgren's Theorem we can say that $u=0$ in $D$. Then the Trace Theorem implies that $f=0$ on $\Gamma_1$, proving injectivity. By the compact embedding of $H^1(D_0)$ into $L^2(D_0)$ implies that $S$ is a compact operator. Proving the claim. 
\end{proof}

In order to complete the factorization we need to define the middle operator $T$ as in the theory developed in Section \ref{RFM}. To this end, notice that by \eqref{prob2} we have that 
$$ -\Delta w +\chi({D}_0) w =  \chi({D}_0) \big[ h + w\big|_{D_0} \big]  \,\, \text{ in } \,\,  D \quad \text{ and } \quad w |_{_{\Gamma_{1}} } = 0$$
and by the definition of $G$ we have that 
$$\dnu w|_{_{\Gamma_{1}} }= G h \quad \text{ as well as} \quad  \dnu w|_{_{\Gamma_{1}} }= -S^*\big[h + w\big|_{D_0}\big]$$
from Theorem \ref{S-adjoint}. Motivated by this we define the operator 
$$T: L^2(D_0) \longrightarrow  L^2(D_0)  \quad \text{ given by } \quad  Th = h + w\big|_{D_0}.$$
By the well-posedness of \eqref{prob2} we have that $T$ defines a bounded linear operator. We observe that one can factorize the operator $G$ such that $G = -S^*T$. Recall, that we have already shown $(\Lambda - \Lambda_0)  =-GS$ which give the following result. 
\begin{theorem}\label{factorization}
The difference of the DtN mappings $(\Lambda - \Lambda_0) : H^{1/2}(\Gamma_{1}) \longrightarrow H^{-1/2}(\Gamma_{1}) $ has the factorization $(\Lambda - \Lambda_0)= S^* T S.$
\end{theorem}

In order to apply Theorem \ref{reg-range-lemma} to solve the inverse problem of recovering $D_0$ from the difference of the DtN mappings $(\Lambda - \Lambda_0)$ we need to prove that $T$ is coercive on the range of $S$ and to characterize the region $D_0$ by the range of $S^*$. Once we have this we can apply Theorem \ref{reg-range-lemma} to reconstruct $D_0$ from the measured Neumann data $\Lambda f = \dnu u|_{_{\Gamma_{1}} } $ and computed Neumann data $\Lambda_0 f = \dnu u_0|_{_{\Gamma_{1}} } $ on the known outer boundary. We now prove coercivity of the operator $T$.

\begin{theorem}\label{coercive}
The operator 
$$T: L^2(D_0) \longrightarrow  L^2(D_0)  \quad \text{ given by } \quad  Th = h + w\big|_{D_0}$$
where $w$ satisfies \eqref{prob2} is coercive on $L^2(D_0)$.
\end{theorem}
\begin{proof}
To prove the claim, notice that 
\begin{align*} 
(Th,  h)_{L^2(D_0)}  = \int\limits_{D_0} |h|^2 +w\overline{h} \,  \text{d}x = \int\limits_{D_0} |h|^2 \, \text{d}x  +  \int\limits_{D} \chi({D}_0) w \overline{h} \, \text{d}x. 
\end{align*}
Now, by equation \eqref{prob2} and Green's 1st Theorem we conclude that 
\begin{align*} 
(Th,  h)_{L^2(D_0)}  = \int\limits_{D_0} |h|^2 \, \text{d}x  -  \int\limits_{D} w\Delta \overline{w} \, \text{d}x  = \int\limits_{D_0} |h|^2 \, \text{d}x  +  \int\limits_{D} |\grad  w |^2\, \text{d}x 
\end{align*}
which prove that claim. 
\end{proof}

From this, we can prove the following result that characterizes the analytical properties for the difference of the DtN mappings.

\begin{theorem}\label{current-gap}
The difference of the DtN mappings $(\Lambda - \Lambda_0) : H^{1/2}(\Gamma_{1}) \longrightarrow H^{-1/2}(\Gamma_{1}) $ is compact and injective with a dense range.
\end{theorem}
\begin{proof}
The compactness and injectivity are clear consequences of Theorems \ref{S-adjoint} and \ref{coercive}. 
In order to prove the density of the range it is sufficient to show that the set of annihilators for Range$(\Lambda - \Lambda_0)$ is trivial(see Corollary 1.8 of \cite{Brezis}). To this end, notice that for all $f,g \in H^{1/2}(\Gamma_{1})$
\begin{align*}
\big\langle g , (\Lambda - \Lambda_0) f \big\rangle_{\Gamma_1}&= \int\limits_{\Gamma_1} g\,  \partial_{\nu} \overline{u^f} -  g \, \partial_{\nu} \overline{u^f_0} \, \text{d}s = \int\limits_{\Gamma_1} u^g \partial_{\nu} \overline{u^f} \, - \,u^g_0 \partial_{\nu} \overline{u^f_0} \, \text{d}s
\end{align*}
Here the pairs of functions $(u^f , u^g) $ and  $(u^f_0 , u^g_0)$ are the solutions to \eqref{prob1} and \eqref{harmonic}, for either $f$ or $g$ respectively.  We appeal to Green's 1st Theorem to obtain
\begin{align*}
\big\langle g , (\Lambda - \Lambda_0) f \big\rangle_{\Gamma_1}&= \int\limits_{D} \grad u^g  \cdot \grad \overline{u^f} + \chi(D_0) u^g \overline{u^f} \, \text{d}x  - \int\limits_{D} \grad u^g_0  \cdot \grad \overline{u^f_0} \, \text{d}x 
\end{align*}
by using \eqref{prob1} and \eqref{harmonic}. Now let $f \in H^{1/2}(\Gamma_{1})$ be an annihilator for Range$(\Lambda - \Lambda_0)$. Then we have that 
\begin{align*}
0=\big\langle f , (\Lambda - \Lambda_0) f \big\rangle_{\Gamma_1}&= \int\limits_{D} |\grad u^f|^2 \, \text{d}x - \int\limits_{D} |\grad u^f_0|^2 \, \text{d}x  + \int\limits_{D_0} |u^f|^2  \, \text{d}x \geq \int\limits_{D_0} |u^f|^2  \, \text{d}x
\end{align*}
where we have used that $u^f_0$ satisfying \eqref{harmonic} minimizes the Dirichlet energy. By Theorem \ref{S-adjoint} we have the $f=0$, which proves the claim.
\end{proof}

Notice, that by Theorems \ref{S-adjoint}, \ref{factorization}, and \ref{coercive} we have that the difference of the DtN mappings $(\Lambda - \Lambda_0)$ satisfies the assumptions of Theorem \ref{reg-range-lemma}. The last piece that we need is to connect the domain $D_0$ to the range of the operator $S^*$. This is due to Theorem  \ref{reg-range-lemma} which gives that 
$$\ell \in\text{Range} (S^*) \quad \text{ if and only if} \quad \liminf\limits_{\alpha \to 0} \langle f_\alpha \, ,(\Lambda - \Lambda_0) f_\alpha \rangle_{\Gamma_1} < \infty$$
where $f_\alpha$ is the regularized solution to $(\Lambda - \Lambda_0) f = \ell$. Since $(\Lambda - \Lambda_0)$ is compact and injective with a dense range we can apply either Tikhonov's regularization or the Spectral cutoff as the regularization scheme as discussed in the previous section. Now, we connect the range of the operator $S^*$ to the unknown region  $D_0$. To do so, we now define the Dirichlet Green's function the negative Laplacian for the known domain $D$ as $\mathbb{G} (\cdot \, ,z) \in H^1_{loc} \big(D \setminus \{z\} \big)$, which is the unique solution to 
$$ - \Delta  \mathbb{G} (\cdot \, , \, z) = \delta(\cdot - z)  \quad \text{in} \quad  D \quad  \text{and} \quad  \mathbb{G} (\cdot \, ,  z)|_{_{\Gamma_{1}} }  =0$$
for any fixed $z \in D$. The idea of the following result is to show that due to the singularity at $z$ the normal derivative of the Green's function is not contained in the range of $S^*$ unless the singularity is contained within the region of interest $D_0$.

\begin{theorem}\label{range}
The operator $S^*$ is such that for any $z \in D$ 
$$\partial_{\nu} \mathbb{G} (\cdot \, ,  z)|_{_{\Gamma_{1}} }  \in \text{Range}(S^*) \quad \text{ if and only if } \quad z \in D_0. $$ 
\end{theorem}
\begin{proof}
To begin, we first assume that $z \in D_0$ which implies that there is some $\ep>0$ such that the ball centered at $z$ satisfies $B(z \, ; \ep) \subset D_0$. Now let $\Psi \in C^{\infty}(\R)$ be a smooth cutoff function with $\Psi(t)=1$ for $|t|\geq \ep$ and $\Psi(t)=0$ for $|t|< \ep/2$ where we define the smooth function $v_z = - \Psi(| \cdot \, - z|)  \mathbb{G} (\cdot \, , \, z)$. Notice that the function $v_z$ satisfies $v_z|_{_{\Gamma_{1}} }=0$ and $-\partial_{\nu} v_z |_{_{\Gamma_{1}} } = \partial_{\nu} \mathbb{G} (\cdot \, ,  z)|_{_{\Gamma_{1}} }$. 
Also, notice that 
$$-\Delta v_z =0 \,\, \text{ in } \,\, D \setminus \overline{D}_0  \quad \text{ and } \quad  -\Delta v_z+\chi({D}_0) v_z \in L^2(D)  \,\, \text{ with support in $D_0$.} $$
Therefore, if we take $g_z =  -\Delta v_z+\chi({D}_0) v_z $ we have that $S^* g_z = \partial_{\nu} \mathbb{G} (\cdot \, ,  z)|_{_{\Gamma_{1}} }$. Proving that 
$$\partial_{\nu} \mathbb{G} (\cdot \, ,  z)|_{_{\Gamma_{1}} }  \in \text{Range}(S^*) \quad \text{ provided} \quad z \in D_0 . $$

Now, we assume that $z \in  D \setminus \overline{D}_0$ and we proceed by way of contradiction. To this end, assume there exists a $g_z \in L^2(D_0)$ such that  $S^* g_z = \partial_{\nu} \mathbb{G} (\cdot \, ,  z)|_{_{\Gamma_{1}} }$. This would imply that there is a $v_z \in H^1(D)$ such that 
$$ -\Delta v_z = 0 \,\, \text{ in } \,\, D \setminus \overline{D}_0 \quad \text{ and } \quad  v_z |_{_{\Gamma_{1}} } = 0 \quad \text{ with } \quad -\partial_{\nu} v_z |_{_{\Gamma_{1}} } = \partial_{\nu} \mathbb{G} (\cdot \, ,  z)|_{_{\Gamma_{1}} } .$$
Therefore, we define $W_z = v_z + \mathbb{G} (\cdot \, ,  z)$ and notice that 
$$\Delta W_z =0 \,\, \text{ in } \,\, D \setminus (\overline{D}_0 \cup \{z\}) \quad \text{ and } \quad W_z  |_{_{\Gamma_{1}} }=\partial_{\nu} W_z |_{_{\Gamma_{1}} } = 0.$$
By appealing to Holmgren's Theorem we can conclude that $W_z =0$ in $D \setminus (\overline{D}_0 \cup \{z\})$ i.e. $v_z= -\mathbb{G} (\cdot \, ,  z)$ in  $D \setminus (\overline{D}_0 \cup \{z\})$. By interior elliptic regularity we have that $v_z$ is continuous at $z \in D \setminus \overline{D}_0$ where as $\mathbb{G} (\cdot \, ,  z)$ has a singularity at $z$. Proving the claim by contradiction since $|v_z(x)|=\mathcal{O}(1)$ and $|\mathbb{G} (x ,  z)| \to \infty$ as $x \to z$. 
\end{proof}

Now, by combining the previous results in this section we have the follow regularized variant of the Factorization Method for recovering the unknown region $D_0$ from the knowledge of difference of the DtN mappings $(\Lambda - \Lambda_0)$. 

\begin{theorem}\label{recon-thm}
The difference of the DtN mappings $(\Lambda - \Lambda_0) : H^{1/2}(\Gamma_{1}) \longrightarrow H^{-1/2}(\Gamma_{1}) $ uniquely determines $D_0$ such that for any $z \in D$
$$ z \in D_0 \quad \text{ if and only if } \quad \liminf\limits_{\alpha \to 0} \big\langle f^z_\alpha \, ,(\Lambda - \Lambda_0) f^z_\alpha \big\rangle_{\Gamma_1} < \infty$$
where  $f^z_\alpha$ is the regularized solution to $(\Lambda - \Lambda_0) f^z = \partial_{\nu} \mathbb{G} (\cdot \, ,  z)|_{_{\Gamma_{1}} }$. 
\end{theorem}

This not only give a theoretically rigorous proof for the inverse shape problem but also a computable algorithm for reconstructing $D_0$. By analyzing the addition of the  regularization step theoretically justifies the uses of these techniques as is done in the literature(see for e.g. \cite{fm-gbc,Liem}). We will see that Theorem \ref{recon-thm} can be used to numerically recover $D_0$.

\section{Numerical Examples}\label{numerics}
In this section, we present numerical examples for the Regularized Factorization Method for solving the inverse shape problem studied in Section \ref{InvShape}. The computations presented here are done in MATLAB 2018a on an iMac with a 4.2 GHz Intel Core i7 processor with 8GB of memory. For simplicity we will consider the problem in $\R^2$ where $D$ is the unit disk. Therefore, we first note that the normal derivative of the Green's function in the unit disk with zero trace is given by the Poisson kernel  
$$  \partial_{\nu} \mathbb{G} \big((1,\theta ),  z \big) |_{_{\Gamma_{1}} } = \frac{1}{2\pi} \left[ \frac{1-|z|^2}{|z|^2 +1-2|z| \cos(\theta - \theta_z )} \right].$$
Here the quantity $\theta_z$ is the polar angle that the sampling point $z \in D$ makes with the positive $x$-axis. 

In order to apply Theorem \ref{recon-thm} to solve the inverse problem we need to compute the difference of the DtN mappings $(\Lambda - \Lambda_0)$. To this end, we will compute the mappings $f \longmapsto (\Lambda - \Lambda_0)f$ where $f$ is given by the standard trigonometric polynomial basis functions. This is due the the fact that the trigonometric polynomials are dense in $H^{ 1/2}(\Gamma_1)=H^{ 1/2}_{\text{per}}[0,2 \pi]$. Recall, that for any $f$ we have that 
\begin{align}
-\Delta (u-u_0)  + \chi({D}_0)  (u-u_0) = - \chi(D_0) u_0  \,\, \text{ in } \,\, D \quad \text{ and } \quad (u-u_0) |_{_{\Gamma_{1}} } = 0. \label{direct-solve}
\end{align}
Since we have assumed that $D$ is the unit disk we have that $u_0$ can be computed via separation of variables. In order to solve \eqref{direct-solve} for $u-u_0$ we consider the variational formulation with the Dirichlet Spectral-Galerkin Method(see \cite{ziteHarris} for details on the approximation space). Once we have computed $u-u_0$ we can then obtain $(\Lambda - \Lambda_0)f$ by take a derivative with respect to $r$ at $r=1$. 

\subsection{Reconstruction by the Regularized Factorization Method}
Here we let $\theta_j $ be $40$ uniformly distributed points in the interval $[0, 2\pi]$. We denote the discretized operator $(\Lambda - \Lambda_0 )$ by the matrix ${\bf A}$ and the vector ${\bf b}_z= \big[ \partial_{r} \mathbb{G} \big( (1,\theta_j ) ,  z \big)  \big]^{40}_{j=1}$ denotes the discretized normal derivative of the Green's function. The sampling points $z$ are given by a uniform grid points in the square $[-1,1]^2$. In our examples, we add random noise $\delta>0$ to the discretized operator such that 
$${\bf A}^\delta =\Big[{\bf A}_{i,j}\big( 1 +\delta \, {\bf E}_{i,j} \big) \Big]_{i,j=1}^{40} \quad \text{ with } \quad  \| {\bf E} \|_2 = 1.$$
Here the matrix ${\bf E}$ is normalized where the initial entries are randomly chosen to be uniformly distributed between $[-1,1]$. To apply Theorem \ref{recon-thm} we numerically solve 
$$ {\bf A}^\delta {\bf f}^z = {\bf b}_z \quad \text{ via a Spectral cutoff.} $$
The regularized solution will be denoted ${\bf f}^z_\alpha$ where the cutoff parameter $\alpha$ is taken to be $10^{-5}$ ad-hoc which gives good results in our experiments. Notice that the discretized version of Theorem \ref{recon-thm} suggests that  
{\color{black} $$ {\big( {\bf f}^z_\alpha , {\bf A}^\delta {\bf f}^z_\alpha \big)} =\mathcal{O}(1)  \quad \text{ if and only if }  \quad z \in D_0$$
when the regularization parameter $\alpha \to 0$. Some simple calculations give that 
$${\big( {\bf f}^z_\alpha , {\bf A}^\delta {\bf f}^z_\alpha \big)} = \sum \frac{\phi^2(\sigma_j  ; \alpha)}{\sigma_j} \big|({\bf u}_j , {\bf b}_z)\big|^2.  $$ Here we let $\phi(t ; \alpha)$ be the filter function describing the regularization method. Also, we will let $\sigma_j$ be the singular values and ${\bf u}_j$ denotes the left singular vectors of the matrix ${\bf A}^\delta$. 
Therefore, in our examples to visualize the region $D_0$ we plot the corresponding imaging functional $W_{\text{RegFM}}(z)$ given by 
$$W_{\text{RegFM}}(z)= \left[ \sum \frac{\phi_{\text{Cutoff}}^2(\sigma_j  ; \alpha)}{\sigma_j} \big|({\bf u}_j , {\bf b}_z)\big|^2 \right]^{-1} \quad \text{with} \quad \phi_{\text{Cutoff}}(t; \alpha)= \left\{\begin{array}{lr} 1 \, &  \, t^2\geq \alpha,  \\
 & \\
 0\, & \,   t^2 < \alpha. \end{array} \right.$$
which is normalized to have maximal value equal to one. This expression of the imaging functional can be used even if the resulting noisy data matrix is not positive definite.} The sampling points are taken to be uniformly space in a grid contained in the unit circle. By Theorem \ref{recon-thm} we expect that $W_{\text{RegFM}}(z)\approx 1$ for $z \in D_0$ and $W_{\text{RegFM}}(z)\approx 0$ for $z \notin D_0$. 

\begin{figure}[ht]
\hspace{-0.5in}\includegraphics[width=18cm]{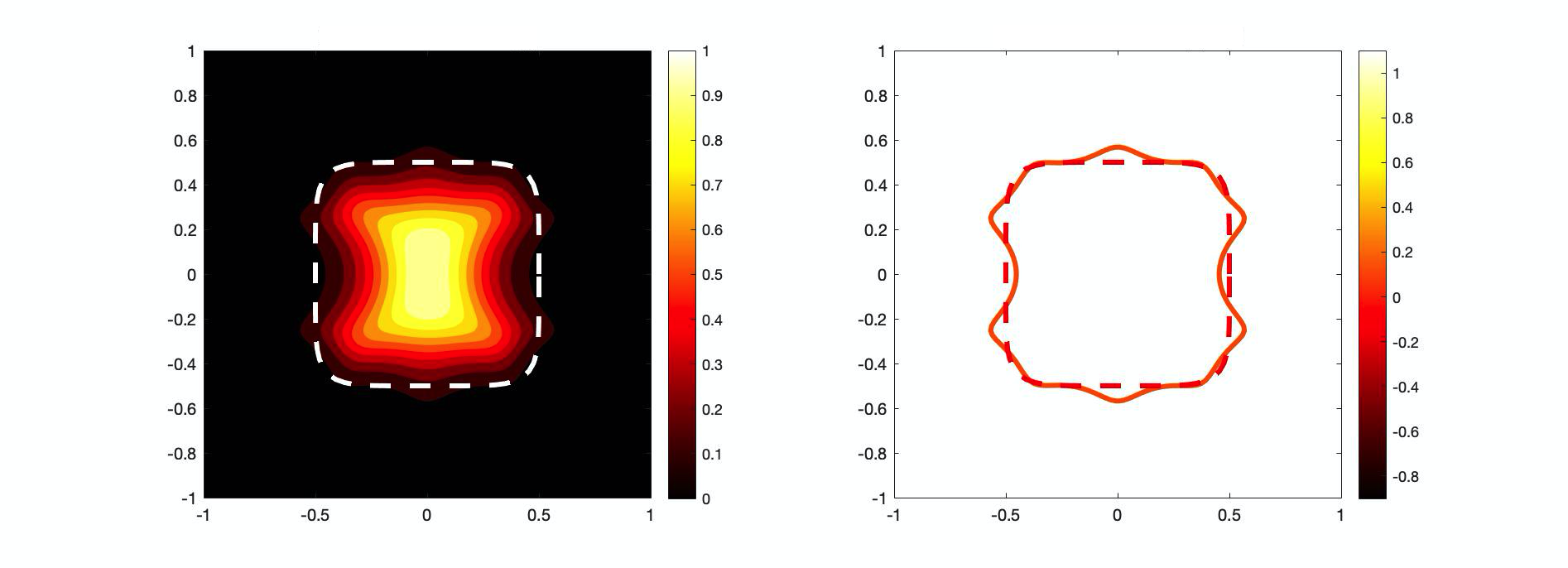}
\caption{Reconstruction of the rounded-square domain by the Regularized FM with noise level $\delta=2\%$. The solid line is the reconstruction of the boundary.}
\label{recon0}
\end{figure}

\begin{figure}[ht]
\hspace{-0.5in}\includegraphics[width=18cm]{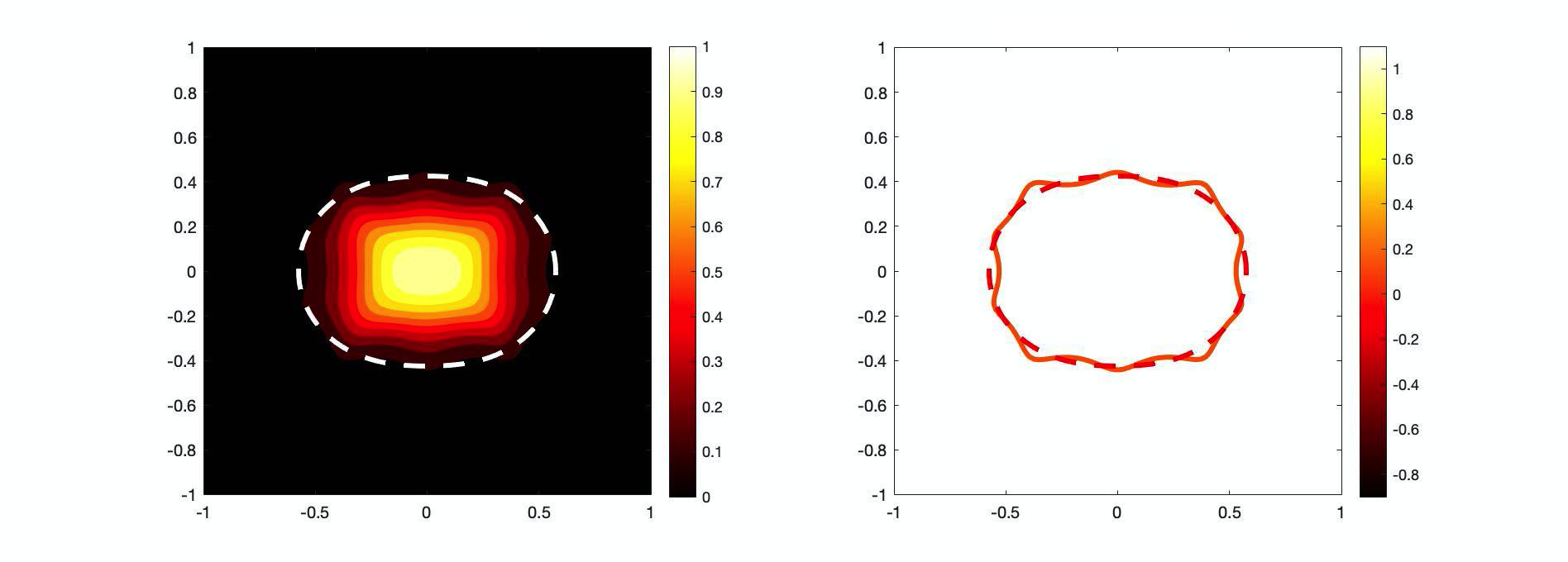}
\caption{Reconstruction of the elliptical shaped domain by the Regularized FM with noise level $\delta=5\%$. The solid line is the reconstruction of the boundary.}
\label{recon1}
\end{figure}

In the examples given we take $D_0$ to be star-like with respect to the origin for simplicity. Then the boundary of $D_0$ is given in a polar representation such that 
$$\partial D_0 =\big\{ \rho(\theta) (\cos\theta, \sin\theta ) \,  \textrm{ where } \, 0\leq \theta < 2\pi \big\}$$ 
with $0<\rho(\theta) < 1$ is a $2\pi$-periodic function. Here we consider different shaped regions given by a rounded-square, elliptical, pear and star shaped region where 
\begin{eqnarray*}
\rho(\theta) = 0.5 \left( |\sin(\theta)|^6 +  |\cos(\theta)|^6 \right)^{-1/6} 
 \quad \text{ or } \quad \rho(\theta) = 0.25 \big(2+0.3\cos(q\theta) \big)
\end{eqnarray*}
with the parameter $q=2,3,5$. Here for $q=2$ corresponds to the elliptical shaped region, $q=3$ corresponds to the pear shaped region, and $q=5$ corresponds to the star shaped region.

\begin{figure}[H]
\hspace{-0.5in}\includegraphics[width=18cm]{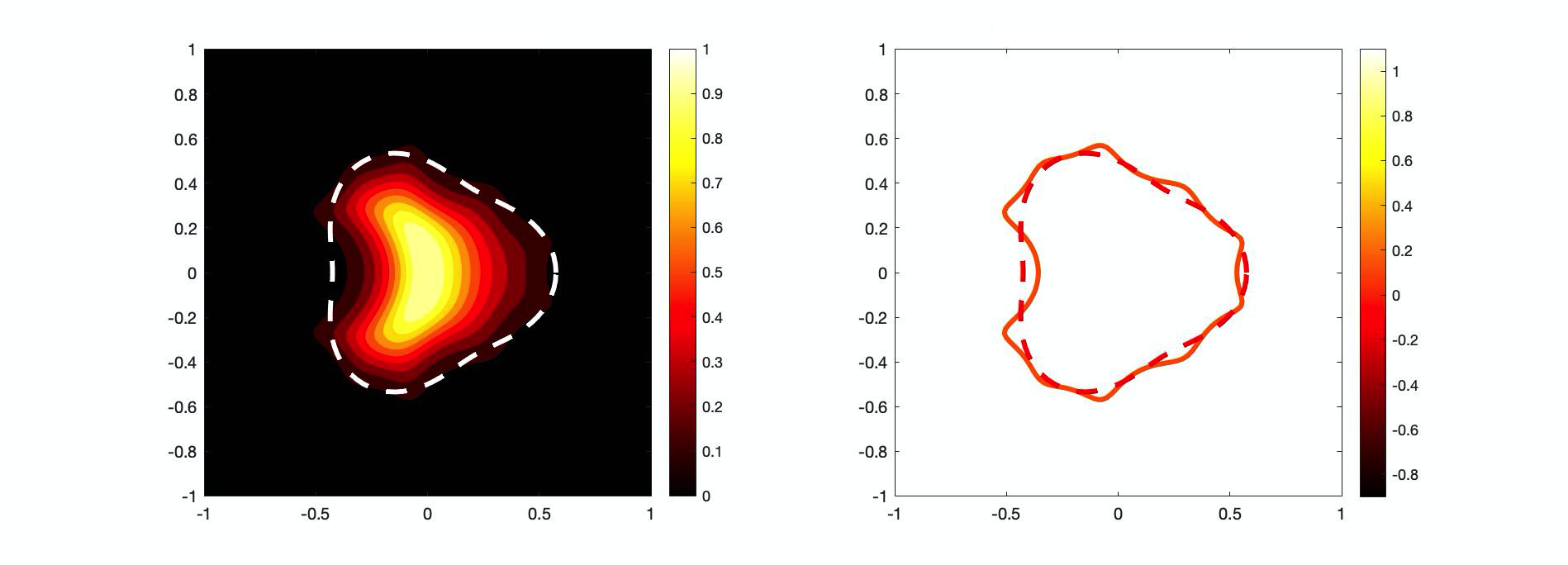}
\caption{Reconstruction of the pear shaped domain by the Regularized FM with noise level $\delta=5\%$. The solid line is the reconstruction of the boundary.   }
\label{recon2}
\end{figure}

\begin{figure}[H]
\hspace{-0.5in}\includegraphics[width=18cm]{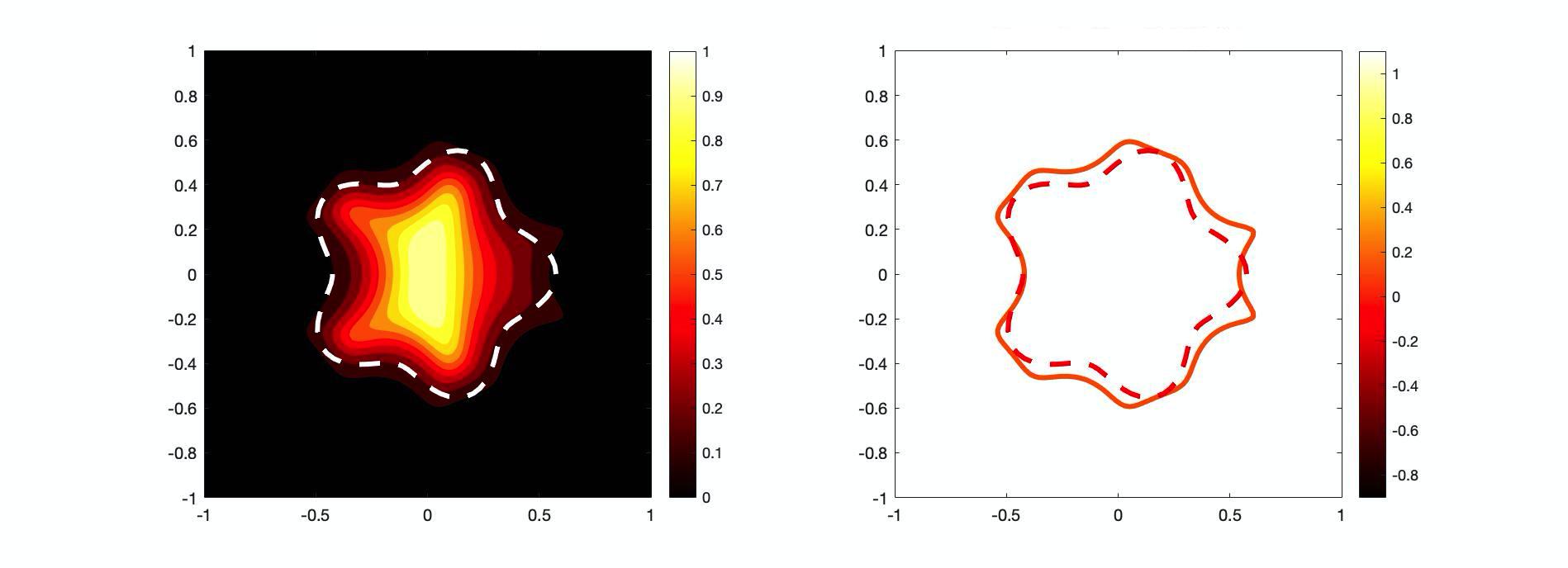}
\caption{Reconstruction of the star shaped domain by the Regularized FM with noise level $\delta=5\%$. The solid line is the reconstruction of the boundary.}
\label{recon3}
\end{figure}

In Figures \ref{recon0}, \ref{recon1}, \ref{recon2}, and \ref{recon3} we present the numerical reconstructions of the different shaped regions. Here the dashed lines are exact boundaries of the region $D_0$. In each reconstruction, we have presented both a contour plot of $W_{\text{RegFM}}(z)$ as well as the contour line $W_{\text{RegFM}}(z)=1/10$ to approximate the boundary $\partial D_0.$ In each of the examples we see that the reconstructions are favorable for each of the shapes. 

\subsection{Comparison to the Generalized Linear Sampling Method}
We now compare our new Regularized Factorization Method(FM) studied here to the Generalized Linear Sampling Method(GLSM). {\color{black}  Recall that the GLSM, gives the same result as Theorem \ref{recon-thm} but the regularization scheme corresponds to minimizing a functional similar to Tikhonov's regularization}. For the problem considered here we need to compute the `minimizer' ${\bf g}^z_\alpha$ of the functional
$$\mathcal{J}_{\alpha} \big({\bf b}_z ; \, {\bf g} \big)  = \alpha \big|\big( {\bf A}^\delta  \, {\bf g} , {\bf g}  \big)  \big| +\|  {\bf A}^\delta {\bf g}- {\bf b}_z \|^2.$$
Therefore, Theorem 3 of \cite{GLSM} suggests that 
$$ {\big|\big( {\bf A}^\delta  \, {\bf g}^z_\alpha, {\bf g}^z_\alpha  \big)  \big|}  = \mathcal{O}(1)  \quad \text{ if and only if } \quad z \in D_0$$
when the regularization parameter $\alpha \to 0$ provided that there is a $C>0$ independent of $\alpha$ such that
$$ \mathcal{J}_{\alpha} \big({\bf b}_z ; \,{\bf g}^z_\alpha \big)  \leq  \inf\limits_{{\bf g}} \mathcal{J}_{\alpha} \big({\bf b}_z ; \, {\bf g} \big)  +C \alpha . $$
Notice, that the result for the GLSM  also needs the extra assumption on the regularized solution given above where as the Regularized FM removes this assumption. It can be shown that when ${\bf A}^\delta$ is a positive definite matrix then the minimizer ${\bf g}^z_\alpha$ satisfies the matrix equation
$$\big(\alpha {\bf A}^\delta +  ({\bf A}^\delta)^{*} {\bf A}^\delta \big) {\bf g}^z_\alpha=({\bf A}^\delta)^{*} {\bf b}_z\quad \text{ for any }\quad \alpha>0.$$
In order to compute ${\bf g}^z_\alpha$, we solve the above equation where the parameter $\alpha$ is taken to be $10^{-3}$ ad-hoc. In general, $\alpha$ should be chosen by a discrepancy principle but in our experiments this choice of $\alpha$ seems to work well. Similarly, the imaging functional for the region $D_0$ for the GLSM is given by $W_{\text{GLSM}}(z)$ where 
$${\color{black}  W_{\text{GLSM}}(z)=\left[ \sum \frac{\phi_{\text{GLSM}}^2(\sigma_j  ; \alpha)}{\sigma_j} \big|({\bf u}_j , {\bf b}_z)\big|^2 \right]^{-1} \quad \text{with filter} \quad \phi_{\text{GLSM}}(t  ; \alpha)=\frac{t}{\alpha+t }}$$
which is normalized to have maximal value equal to one. In Figures \ref{compair1} and \ref{compair2}, we plot the imaging functionals for both the GLSM and Regularized FM for both the rounded-square and star shaped inclusions. 
We can see that both methods give similar reconstructions of the unknown region.

\begin{figure}[H]
\hspace{-0.5in}\includegraphics[width=18cm]{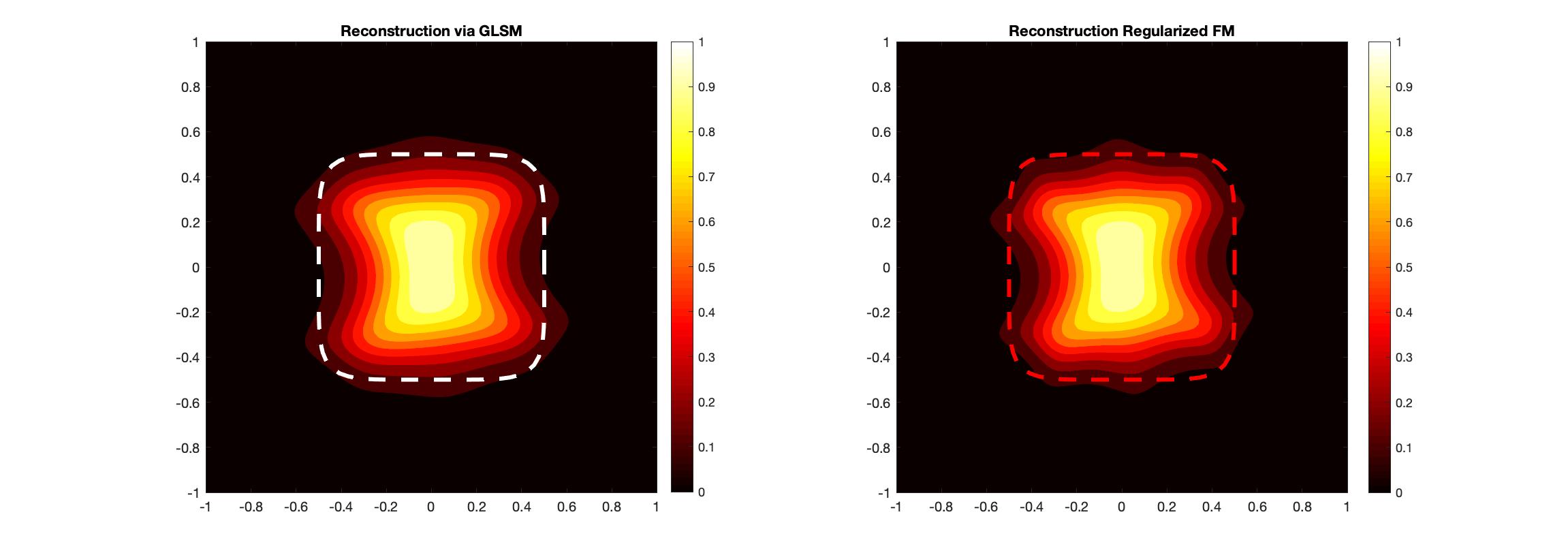}
\caption{Reconstruction by the GLSM on the left and the Regularized FM on the right for the rounded-square domain with noise level $\delta=2\%$. }
\label{compair1}
\end{figure}

\begin{figure}[H]
\hspace{-0.5in}\includegraphics[width=18cm]{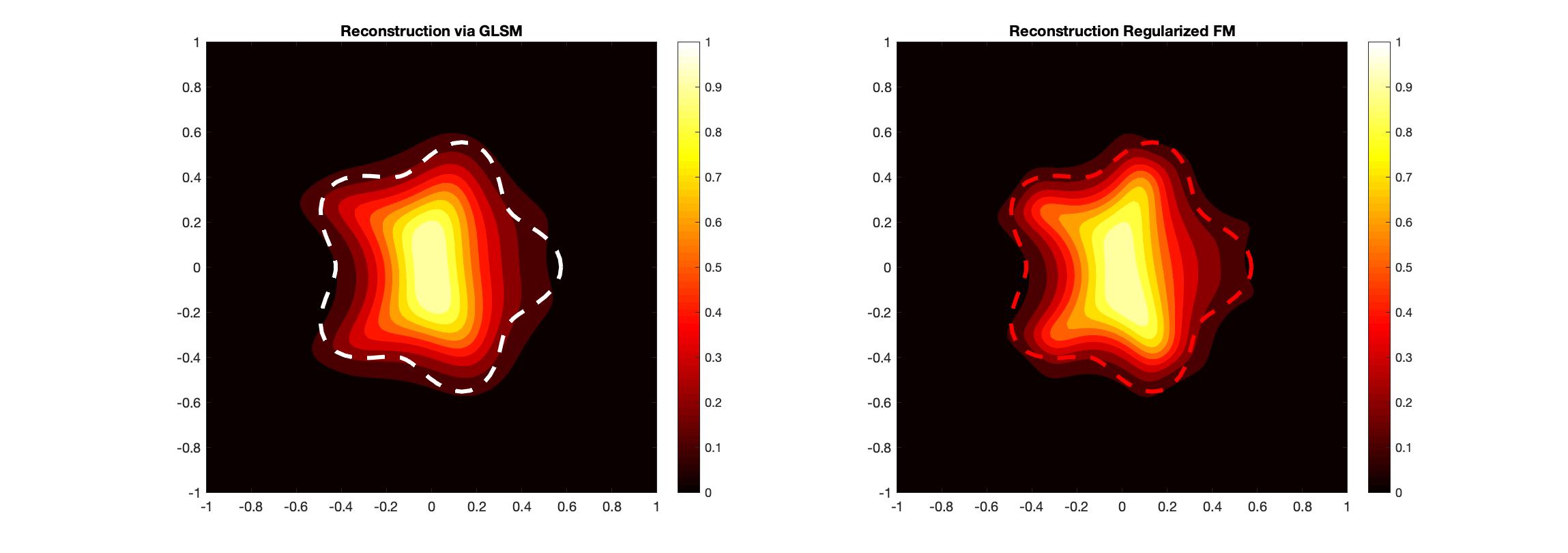}
\caption{Reconstruction by the GLSM on the left and the Regularized FM on the right for the star shaped domain with noise level $\delta=5\%$.  }
\label{compair2}
\end{figure}

\section{Conclusions}\label{end}
In conclusion, we have analyzed a regularized version of the FM and applied our analysis to solve an inverse shape problem coming form semiconductor theory. The result that we have proven matches the regularization techniques that are used in the literature, see for e.g. \cite{fm-gbc,Liem} in the case of inverse scattering. The analysis here generalizes result in \cite{Harris-Rome} to a generic positive compact operator  $A: X \to X^*$ where $X$ is a complex Hilbert Space. The analysis here can be used to recover other inclusions in optical tomography such as the problem discussed in \cite{DOT-RtR,RtR}. One main advantage that this Regularized FM has over the GLSM is that one does not have to minimize a possible non-convex functional. The method presented here can be used for application in inverse scattering. This method is analytically rigorous and gives good numerical reconstructions. Another thing to note is that the GLSM has been shown to work in the presence of noise in the data where this Regularized FM has not been verified for noise present in the operator $A$. 
One would expect that a for $A^{\delta}: X \to  X^*$ such that $\|A^{\delta} - A\| \to 0$ as $\delta \to 0$ then
$$ \ell \in \text{Range}(S^*) \quad \text{ if and only if} \quad \liminf\limits_{\alpha \to 0} \liminf\limits_{\delta \to 0}  \big\langle x^{\delta}_\alpha, A^{\delta} x^{\delta}_\alpha \big\rangle_{X \times X^*} < \infty$$ 
with $x^{\delta}_\alpha$ the regularized solution to $A^{\delta} x=\ell$. Here it is assumed that the regularization parameter $\alpha$ would need to be chosen via a discrepancy principle. For this case we would assume that $\alpha(\delta) \to 0$ as the error $\delta \to 0$ which is the case for Morozov discrepancy principle. The study of the Regularized FM for a perturbed operator will be studied in future works. \\

\noindent{\bf Acknowledgments:} The research of I. Harris is partially supported by the NSF DMS Grant 2107891.


\end{document}